\documentclass[a4paper,11pt,leqno]{amsart}
\usepackage{amsmath}
\usepackage{amsthm}
\usepackage{amsfonts}
\usepackage{amssymb}

\usepackage[usenames]{color}
\usepackage[all]{xy}

\usepackage{url}
\usepackage[all]{xy}
\usepackage{graphicx}
\usepackage{latexsym}
\usepackage[cp850]{inputenc}
\usepackage{epsfig}
\usepackage{psfrag}
\usepackage{amsfonts}
\usepackage{dsfont}

\newtheorem{theorem}{Theorem}[section]{\bf}{\it}
\newtheorem{lemma}[theorem]{Lemma}{\bf}{\it}
\newtheorem{proposition}[theorem]{Proposition}{\bf}{\it}
\newtheorem{corollary}[theorem]{Corollary}{\bf}{\it}
{\bf}{\it} 
{\bf}{\it}
\newtheorem*{theorem*}{Theorem}

\newtheorem*{namedtheorem}{\theoremname}
\newcommand{\theoremname}{testing}
\newenvironment{named}[1]{\renewcommand{\theoremname}{#1}\begin{namedtheorem}}{\end{namedtheorem}}

{\bf}{\it}
{\bf}{\it}

\theoremstyle{remark}
\newtheorem*{remark}{Remark}

\newtheorem{claim}{Claim}

\theoremstyle{definition}
\theoremstyle{remark}

\numberwithin{equation}{section}

\newcommand{\R}{\mathbb R}
\newcommand{\Z}{\mathbb Z}

\newcommand{\loc}{{\operatorname{loc}}}

\newcommand{\id}{{\operatorname{id}}}

\newdimen\vintkern\vintkern11pt
\def\vint{-\kern-\vintkern\int}

\newcommand{\bS}{\mathbb{S}}

\newcommand{\bT}{\mathbb{T}}
\newcommand{\bF}{\mathbb{F}}

\newcommand{\Aut}{\mathrm{Aut}}
\newcommand{\Ker}{\mathrm{Ker}}
\newcommand{\Fix}{\mathrm{Fix}}

\begin{document}

\title{On the non-existence of certain branched covers}
\date{\today}
\author{Pekka Pankka and Juan Souto}

\thanks{The first author has been supported by Academy of Finland project \#1126836 and by NSF grant DMS-0757732.  The second author has been partially supported by NSF grant DMS-0706878, NSF Career award 0952106 and the Alfred P. Sloan Foundation}

\begin{abstract}
We prove that while there are maps $\bT^4\to\#^3(\bS^2\times\bS^2)$ of arbitrarily large degree, there is no branched cover from $4$-torus to $\#^3(\bS^2\times \bS^2)$. More generally, we obtain that, as long as $N$ satisfies a suitable cohomological condition, any $\pi_1$-surjective branched cover $\bT^n \to N$ is a homeomorphism.
\end{abstract}


\maketitle

\section{Introduction}

In this paper we will prove the following statement.

\begin{theorem}\label{no-sudoku}
There is no branched cover from the 4-dimensional torus $\bT^4$ to $\#^3(\bS^2\times\bS^2)$, the connected sum of three copies of $\bS^2\times\bS^2$. On the other hand, there are, a fortiori $\pi_1$-surjective, maps $\bT^4\to \#^3(\bS^2\times\bS^2)$ of arbitrarily large degree.
\end{theorem}

Recall that a branched cover between closed PL-manifolds is a discrete and open PL-map; in the sequel we work always in the PL-category and consider only oriented manifolds and orientation preserving maps; this is no reduction of generality as the torus admits an orientation reversing involution.

Starting with the classical result of Alexander \cite{Alexander}, asserting that every closed, orientable PL-manifold of dimension $n$ admits a branched cover onto the $n$-dimensional sphere $\bS^n$, a series of authors \cite{Berstein-Edmonds-examples,Edmonds-deg3,Hilden,Hirsch,Iori-Piergallini,Montesinos,Piergallini} have proved the existence of branched covers satisfying certain conditions between certain manifolds. Of these results, the one which is perhaps most relevant to this note, is due to Edmonds \cite{Edmonds-deg3}: \emph{Every $\pi_1$-surjective map $f:M\to N$ between closed, orientable 3-manifolds with $\deg(f)\ge 3$ is homotopic to a branched cover}.
Theorem \ref{no-sudoku} shows that Edmonds's theorem fails in dimension 4.

One of the reasons why we find Edmonds's theorem to be rather surprising, is because it asserts that there are many more branched covers than we would na\"ively have expected. A different interpretation of the same theorem is that it is not clear how to distinguish between branched covers and maps of positive degree. We insist on this point. When trying to rule out the existence of branched coverings between two closed manifolds $M$ and $N$, one can perhaps observe that any branched cover $f\colon M\to N$ induces, via pull-back, an injective homomorphism
\begin{equation}\label{eq-injective}
f^*:H^*(N;\R)\to H^*(M;\R)
\end{equation}
between the associated cohomology rings. Although this poses a huge restriction to the existence of branched covers, the same restriction applies for the existence of mere maps of non-zero degree. In fact, it is due to Duan and Wang \cite{Duan-Wang} that, at least when the target $N$ is a simply connected $4$-dimensional manifold, the existence of maps $M\to N$ of positive degree is equivalent to the existence of maps $H^2(N;\Z)\to H^2(M;\Z)$ scaling the intersection form. In the light of Theorem \ref{no-sudoku}, the existence of branched covers is a more subtle issue which does not seem to be detected by any standard cohomological property.  Cohomological considerations play a central r\^ole in the proof of Theorem \ref{no-sudoku}, but they come into our argument in a perhaps unexpected way.
\medskip

\noindent{\bf Quasiregularly elliptic manifolds; the original motivation.} We would like to mention that our interest to investigate the existence of branched covers from $\bT^4$ onto $\#^3(\bS^2\times\bS^2)$ stems from a question of Gromov \cite{GromovM:Hypmga,GromovM:Metsrn} and Rickman \cite{RickmanS:Exiqm}, who asked if the manifolds $\#^k(\bS^2\times\bS^2)$ are {\em $K$-quasiregularly elliptic} for some $K$.

Recall that a continuous mapping $f\colon M\to N$ between oriented Riemannian $n$-manifolds is said to be \emph{$K$-quasiregular}, with {\em (outer) distortion} $K\ge 1$, if $f$ is a Sobolev mapping in $W^{1,n}_\loc(M;N)$ and satisfies the quasiconformality condition
\[
|Df|^n \le K J_f\quad \mathrm{a.e.},
\]
where $|Df|$ is the operator norm of the differential $Df$ and $J_f$ is the Jacobian determinant. A Riemannian manifold of dimension $n$ is \emph{$K$-quasiregularly elliptic} if it admits a non-constant $K$-quasiregular mapping from $\R^n$. We refer to Rickman \cite{RickmanS:Quam} and Bonk-Heinonen \cite{BonkM:Quamc} for a detailed discussion on quasiregular mappings and quasiregularly elliptic manifolds respectively. 

Returning to the quasiregular ellipticity of manifolds $\#^k (\bS^2 \times \bS^2)$, recall that Bonk and Heinonen showed in \cite{BonkM:Quamc} that $K$-quasiregularly elliptic manifolds have quantitatively bounded cohomology. In particular, there exists a bound $k_0=k_0(K)$ depending only on $K$, so that $\#^k(\bS^2\times \bS^2)$ is not $K$-quasiregularly elliptic for $k>k_0$. 

Without a priori bounds on the dilatation $K$, the question on the quasiregular ellipticity of manifolds $\#^k (\bS^2 \times \bS^2)$ is only answered for $k=1,2$. We sketch the argument used in these two cases. Observing that branched covers between closed manifolds are quasiregular, it follows that every manifold $N$ for which there is a branched cover $f:\bT^4\to N$ is quasiregularly elliptic. Considering $\bT^4$ as the product of two 2-dimensional tori and taking on each factor a branched covering $\bT^2\to\bS^2$, one obtains a branched cover $\bT^4\to\bS^2\times\bS^2$ and deduces that the latter manifold is quasiregularly elliptic. What is really much more surprising, is that there is also a branched cover $\bT^4\to\#^2(\bS^2\times\bS^2)$; see \cite{Rickman} for this construction. In particular $\#^2(\bS^2\times\bS^2)$ is also quasiregularly elliptic.

The unfortunate offshot of Theorem \ref{no-sudoku} is that this argument cannot be applied to $\#^k (\bS^2\times \bS^2)$ for $k=3$. Notice that for $k>3$, this observation follows directly from the injectivity of the homomorphism \eqref{eq-injective}. 

\medskip

\noindent{\bf Sketch of the proof of Theorem \ref{no-sudoku}.}
We will deduce the existence of maps $\bT^4\to\#^3(\bS^2\times\bS^2)$ with arbitrarily large degree from the work of Duan and Wang \cite{Duan-Wang}. The non-existence of branched covers in Theorem \ref{no-sudoku} is a special case of the following more general result:

\begin{theorem}
\label{thm:main}
Let $N$ be a closed, connected, and oriented $n$-manifold, $n\ge 2$ so that $\dim H^r(N;\R)=\dim H^r(\bT^n;\R)$ for some $1\le r < n$.
Then every branched cover $\bT^n \to N$ is a cover. In particular, every $\pi_1$-surjective branched cover $\bT^n \to N$ is a homeomorphism.
\end{theorem}

Suppose that $f\colon \bT^n\to N$ is a branched cover as in the statement of Theorem \ref{thm:main}. By a result of Berstein-Edmonds \cite{Berstein-Edmonds}, there are a compact polyhedron $X$, a finite group $G$ of automorphisms of $X$, and a subgroup $H$ of $G$, such that $N=X/G$, $\bT^n=X/H$ and such that the map $f$ is just the orbit map $X/H\to X/G$. Our first goal is to prove that the group $G$ acts on the $H$-invariant cohomology $H^*(X;\R)^H$ of $X$; this is the content of section \ref{sec:homology} below. 

In section \ref{sec:averaging}, we use the action of $G$ on $H^1(X;\R)^H$ to construct, by averaging, cocycles $\Theta_1,\ldots,\Theta_n$ which form a basis of $H^1(X;\R)^H$ and behave well with respect to the action of $G$ on $C^1(X;\R)$, the vector space of $1$-cochains on $X$. 

These cocycles give rise to a map $X\to\R^n/\Z^n$ which semi-conjugates the action of $G$ on $X$ to a certain action on $\R^n/\Z^n$. This map is defined in the same way as the classical Abel-Jacobi map and this is the way we will refer to it below. The reader can find the construction and some properties of the Abel-Jacobi map in section \ref{sec:abel-jacobi} and section \ref{sec:pushdown}. The Abel-Jacobi map and its $G$-equivariance properties, are the key tools to show that the map $f\colon \bT^n\to N$ factors as $\bT^n\to N'\to N$ where the first arrow is a cover and the second arrow has at most degree 2; moreover, $N'$ is homotopy equivalent to a torus. 

In section \ref{sec:final} we will conclude the proof of Theorem \ref{thm:main} showing that the branched cover $N'\to N$ is a homeomorphism. In order to do so, we have to rule out that it has degree 2. Every degree 2 branched cover is regular, meaning that there is an involution $\sigma:N'\to N'$ such that $N=N/\langle\sigma\rangle$ and that the map $N'\to N=N'/\langle\sigma\rangle$ is the orbit map. At this point we will already know that the involution $\sigma$ acts as $-\id$ on $H_1(N';\R)\simeq\R^n$. The Lefschetz fixed-point theorem implies that  the fixed-point set $\Fix(\sigma)$ is not empty. The final contradiction is obtained when we show, using Smith theory, that $\Fix(\sigma)$ is in fact empty.

Once Theorem \ref{thm:main} is proved, we discuss Theorem \ref{no-sudoku} in section \ref{sec:no-sudoku}. 
\medskip

The first author would like to thank Seppo Rickman for frequent discussions on these topics and encouragement. The second author would like to express his gratitude to the Department of Mathematics and Statistics of the University of Helsinki for its hospitality. The second author also wishes to dedicate this paper to {\em Paul el pulpo}.

\section{Preliminaries}\label{sec:preli}

In this section we remind the reader of a few well-known facts and definitions. Throughout this note we will work in the PL-category. We refer the reader to \cite{Hudson, Rourke-Sanderson} for basic facts on PL-theory. 

\subsection{Chains, cochains, and homology}
Suppose that $X$ is a compact {\em polyhedron}, i.e. the geometric realization of a finite simplicial complex and let $R$ be either $\R$ or $\Z$. We denote by $C_k(X;R)$,  $Z_k(X;R)$, $C^k(X;R)$ and $Z^k(X;R)$ the $R$-modules of singular $k$-chains, $k$-cycles, $k$-cochains, and $k$-cocycles with coefficients in $R$. The associated homology and cohomology groups are denoted, as always, by $H_k(X;R)$ and $H^k(X,R)$. The homology (resp. cohomology) class represented by a chain (resp. cochain) $\alpha$ will be denoted by $[\alpha]$.

If $\phi\colon X\to Y$ is a continuous map to another polyhedron $Y$. We denote by $\phi_\# \colon C_k(X;R)\to C_k(Y;R)$ and by $\phi^\# \colon C^k(Y;R) \to C^k(X;R)$ the push-forward and pull-back of chains and cochains, respectively. On homological level, we denote the push-forward and pull-back by $\phi_*:H_*(X;R)\to H_*(Y;R)$ and $\phi^*\colon H^*(Y;R)\to H^*(X;R)$ respectively. 

Finally, denote by
$$\langle\cdot,\cdot\rangle:C^k(X;R)\times C_k(X;R)\to R$$
the canonical evaluation form. Notice that, with the same notation as above, we have 
\begin{equation}\label{eq-bla}
\langle\phi^\#(\omega),c\rangle=\langle\omega,\phi_\#(c)\rangle
\end{equation}
for $\omega\in C^k(Y;R)$ and $c\in C_k(X;R)$. The induced bilinear form on $H^k(X;R)\times H_k(X;\R)$ is also denoted by $\langle\cdot,\cdot\rangle$; again \eqref{eq-bla} is satisfied
\begin{equation}\label{eq-bla32}
\langle\phi^*([\omega]),[c]\rangle=\langle[\omega],\phi_*([c])\rangle.
\end{equation}
Recall that if $c$ and $\omega$ are a chain and a cochain representing the homology and cohomology classes $[c]$ and $[\omega]$ we have $\langle\omega,c\rangle=\langle[\omega],[c]\rangle$.

Continuing with the same notation, suppose that $G$ is a finite group acting on the polyhedron $X$. We denote by $X/G$ the quotient space and by 
$$\pi_G:X\to X/G$$
the orbit map. The $G$-invariant real cohomology 
$$H^*(X;\R)^G=\{\omega\in H^*(X;\R)\vert g^*(\omega)=\omega\ \hbox{for all}\ g\in G\}$$ 
of $X$ is isomorphic to the real cohomology $H^*(X/G;\R)$ of the orbit space $X/G$. Given that this fact will be used over and over again, we state it as a proposition:

\begin{proposition}\label{prop:invco}
Let $X$ be a polyhedron and $G$ a finite group acting on $X$. The pull-back of the orbit map $\pi_G \colon X\to X/G$ induces an isomorphism between $H^*(X;\R)^G$ and $H^*(X/G;\R)$.\qed
\end{proposition}

Proposition \ref{prop:invco} is a consequence of the natural \emph{transfer homomorphism} $H_*(X/G;\R)\to H_*(X;\R)$, see e.g. \cite[III.2]{Bredon}. Below we will need the following closely related fact, whose proof we leave to the reader:

\begin{lemma}\label{lem:transfer}
Let $X$ be a polyhedron, $G$ a finite group acting on $X$ and suppose that $c$ is a $k$-chain whose push-forward $(\pi_G)_\#(c)$ under the orbit map $\pi_G:X\to X/G$ is a $k$-cycle in $X/G$. Then $\sum_{g\in G}g_\#(c)$ is a $k$-cycle in $X$.

Moreover, if $[\omega]$ is a $k$-cohomology class in $X/G$ then 
$$\left\langle\pi_G^*([\omega]),\left[\sum_{g\in G}g_\#(c)\right]\right\rangle=\vert G\vert\langle[\omega],[(\pi_G)_\#(c)]\rangle.$$
\qed
\end{lemma}

We refer the reader to \cite{Bredon,Bredon-top,Hatcher} for basic facts of algebraic topology.

\subsection{Branched covers}
A {\em branched cover} between oriented PL-manifolds is an orientation preserving discrete and open PL-map. Given any such branched cover $f\colon M\to N$ we say that $x\in M$ is a \emph{singular point of $f$} if $f$ is not a local homeomorphism in a neighborhood of $x$. The image $f(x)$ of a singular point is a {\em branching point} of $f$. We denote by $S_f\subset M$ and $B_f=f(S_f)\subset N$ the sets of singular points and branch points of $f$ respectively. We also denote by $Z_f=f^{-1}(B_f)$ the pre-image of the set of branching points of $f$; notice that $S_f\subset Z_f$. 

Before moving on, observe that since $f$ is PL, discrete and open, the sets $S_f,B_f$ and $Z_f$ are polyhedra of codimension at least $2$. Remark also that the restriction of $f$ to $M\setminus Z_f$ is a covering map onto $N\setminus B_f$. 

\begin{remark}
It is worth noticing that the above statement does not require the mapping to be PL. Indeed, given a discrete and open mapping between $n$-manifolds the singular set has topological codimension at least $2$ by a result of {\v{C}}ernavski{\u{\i}}-V\"ais\"al\"a \cite{CernavskiA:Ftoomm,VaisalaJ:Disomm}. Moreover, $B_f$ and $S_f$ have the same (topological) dimension.
\end{remark}

A branched cover $f\colon M\to N$ is said to be {\em regular} if there is a group $G$ of PL-automorphisms of $M$ such that there is an identification of $N$ with $M/G$ in such a way that the map $f$ becomes the orbit map $\pi_G:M\to M/G=N$. The following result, due to Berstein and Edmonds \cite[Proposition 2.2]{Berstein-Edmonds}, asserts that every branched cover is a, non-regular, orbit map:

\begin{proposition}\label{polyhedron}
Given a branched cover $f\colon M\to N$ between closed orientable manifolds of dimension $n$, there exists a connected polyhedron $X_f$ of dimension $n$, a group $G$ of automorphisms of $X_f$, a subgroup $H$ of $G$ and identifications $M=X_f/H$ and $N=X_f/G$ so that the following diagram commutes:
\[
\xymatrix{
& X_f \ar[dl]_{\pi_H} \ar[dr]^{\pi_G} & \\
X_f/H=M \ar[rr]_f & & N=X_f/G
}
\]
Moreover, $H^n(X_f,\R)\simeq\R$ and $G$ acts trivially on $H^n(X_f,\R)$.
\end{proposition}

We observe that in general the polyhedron $X_f$ is not a manifold. However, the complement of $\pi_G^{-1}(B_f)=\pi_H(Z_f)$ in $X_f$ is a manifold and the restriction of $\pi_G$ and $\pi_H$ to that set is a covering map. Just to clarify terminology, we would like to mention that under a {\em cover} or {\em covering map} we understand, as usual, a branched cover whose singular set is empty. 

Before going any further we state, with the same notation as in Proposition \ref{polyhedron}, the following simple consequence of the fact that $H^n(X_f;\R)=\R$ and that $G$ acts trivially on $H^n(X_f;\R)$. If $K$ is namely a subgroup of $G$ then $H^n(X_f/K;\R)\simeq\R$ by Proposition \ref{prop:invco}. Hence, if $V_1,V_2$ are oriented $n$-manifolds then any two continuous maps
$$f_1:V_1\to X_f/K,\ f_2:X_f/K\to V_2$$
have a well-defined degree $\deg(f_i)$ which satisfies the usual rule 
$$\deg(f_2\circ f_2)=\deg(f_2)\deg(f_1).$$
This fact will be of some importance below.

\subsection{Branched covers of degree 1 or 2}
Below, we will show that any branched cover $f:\bT^n\to N$ as in the statement of Theorem \ref{thm:main} is the composition of a covering map and a branched cover of degree one or two. It is well-known that such branched covers are very particular:

\begin{proposition}\label{prop:deg12}${ }$

\begin{enumerate}
\item A branched cover of degree 1 is a homeomorphism.
\item Let $f:M\to N$ be a degree 2 branched cover between closed, connected, orientable PL-manifolds. Then there is an orientation preserving involution $\sigma:M\to M$ such that $N=M/\langle\sigma\rangle$ and such that the map $f$ is just the natural orbit map $M\to M/\langle\sigma\rangle=N$.\qed
\end{enumerate}
\end{proposition}

In the statement of Proposition \ref{prop:deg12}, as well as in the future, we denote by $\langle\sigma\rangle$ the group generated by $\sigma$. 

We conclude this section with the following observation on the fixed point sets of involutions:

\begin{proposition}\label{fix-set}
Let $M$ be a closed oriented $n$-manifold, $n\ge 3$, and $\sigma\colon M\to M$ an orientation preserving involution such that $M/\langle\sigma\rangle$ is a manifold. Either 
$\sigma$ is fixed-point free or given a component $\Sigma$ of the fixed point set $\Fix(\sigma)$ of $\sigma$, there exists $1\le r\le n-2$ so that $H^r(\Sigma;\Z_2)=\Z_2$. In particular, the components of $\Fix(\sigma)$ are not $\Z_2$-acyclic. 
\end{proposition}

Recall that a space is $\Z_2$-acyclic if it has the same cohomology as a point.

\begin{proof}
Recall that a $(\Z_2,r)$-manifold is a topological space satisfying a few properties, the most important of which is that it admits a cover $\mathcal{B}$ such that the one-point compactification of every element in $\mathcal{B}$ has the same cohomology relative to the point at infinity, as the $r$-dimensional sphere $\bS^r$ relative to one of its points. See  \cite[Definition 1.1]{Bredon1960} for details and notice that it follows directly from the definition that a connected polyhedron which is a $(\Z_2,0)$-manifold is a singleton.

By \cite[Theorem 7.4]{Bredon1960} the components of $\Fix(\sigma)$ are $(\Z_2,r)$-manifolds, with $r\le n-2$ because $\sigma$ preserves the orientation \cite[Theorem 7.8]{Bredon1960}. Moreover, by \cite[Theorem 7.11]{Bredon1960}, there exists a natural isomorphism $H^r(\Sigma;\Z_2) \to H^n(M;\Z_2)$. Thus it suffices to show that $r>0$, that is, $\Fix(\sigma)$ has no isolated points.

Seeking a contradiction, suppose that $x\in M$ is an isolated fixed point of $\sigma$. Choosing a sufficiently fine, $\sigma$-invariant triangulation of $M$, let $L_x$ be the link of $x$ and notice that $L_x$ is $\sigma$-invariant, and that $L_x/\langle\sigma\rangle$ is the link of the projection of $x$ to $M/\langle\sigma\rangle$. The link $L_x$ of $x$ is homeomorphic to $\bS^{n-1}$. Since $x$ is an isolated fixed-point of $\sigma$ we have that $\sigma$ acts on $L_x=\bS^{n-1}$ without fixed points. This implies that $\pi_1(L_x/\langle\sigma\rangle)=\Z_2$. In particular, the projection of $x$ to $M/\langle\sigma\rangle$ is a point whose link does not have the homology of a sphere of dimension $n-1\ge 2$. This contradicts the assumption that $M/\langle\sigma\rangle$ is a manifold.
\end{proof}
\bigskip

At this point we fix some notation that will be used throughout the whole paper:

\begin{quote}
{\bf (*) Notation:} Suppose that $f\colon \bT^n\to N$ is a branched cover as in the statement of Theorem \ref{thm:main}; in other words, we may fix $1\le r < n$ so that $\dim H^r(N;\R)=\dim H^r(\bT^n;\R)$. Let also $X=X_f$ be the polyhedron provided by Proposition \ref{polyhedron} and $H\subset G$ the groups of automorphisms of $X$ with $\bT^n=X/H$ and $N=X/G$. Let also be 
$$\pi=\pi_H\colon X\to X/H=\bT^n$$
be the orbit map onto $X/H$.
\end{quote}

\section{The key observation}\label{sec:homology}

The key tool in the proof of Theorem \ref{thm:main} is the analysis of the action of $G$ on the singular (real) cohomology ring $H^*(X;\R)$ of $X$. 

Recall that if $\Gamma\subset G$ is a subgroup of $G$, then $H^*(X;\R)^\Gamma$ is the $\Gamma$-invariant cohomology of $X$ and that the orbit map $X\to X/\Gamma$ induces, via the pull-back, an isomorphism between $H^*(X/\Gamma;\R)$ and $H^*(X;\R)^\Gamma$.

Since $H$ is a subgroup of $G$, it is immediate that $H^*(X;\R)^G$ is $H$-invariant. The main result of this section is a partial converse.

\begin{proposition}\label{prop:inv_H*}
With the same notation as in (*), the following holds:
\begin{enumerate}
\item elements of $G$ act trivially on $H^r(X;\R)^H$, and
\item the subspace $H^1(X;\R)^H$ is $G$-invariant. Moreover, for any $g\in G$ we have that the restriction of $g^*$ to $H^1(X;\R)^H$ is equal to $\pm\id$.
\end{enumerate}
\end{proposition}

To prove the first claim of Proposition \ref{prop:inv_H*} observe that
\[
\dim H^r(X;\R)^G = \dim H^r(X/G;\R)=\dim H^r(N;\R)
\]
and, similarly,
\[
\dim H^r(X;\R)^H = \dim H^r(X/H;\R)=\dim H^r(\bT^n;\R).
\]

Since $\dim H^r(N;\R)=\dim H^r(\bT^n;\R)$ and $H^r(X;\R)^G\subset H^r(X;\R)^H$, it follows that $H^r(X;\R)^G=H^r(X;\R)^H$. Thus $G$ acts trivially on $H^r(X;\R)^H$ by definition.

The second part of the proof uses the following observation.
\begin{lemma}
\label{lemma:rigidity}
Let $r\ge 1$, $V$ be a vector space, $W$ a subspace of $V$ with $\dim W\ge r+1$, and $\varphi \colon V \to V$ a linear mapping so that $\wedge^r\varphi = \id \colon \bigwedge^r W \to \bigwedge^r W$. Then $W$ is $\varphi$-invariant and $\varphi|W=\pm \id$.
\end{lemma}

\begin{proof}
Since $\wedge^r\varphi=\id$ on $\bigwedge^r W$, all $r$-dimensional subspaces of $W$ are invariant under $\varphi$; hence $W$ is $\varphi$-invariant. Let $v\in W$, $v\neq 0$, and let $U_1,\ldots, U_s$ be $r$-dimensional subspaces of $W$ whose intersection $L=\cap U_i$ is the line containing $v$. Since subspaces $U_i$, $i=1,\ldots, s$, are $\varphi$-invariant, the line $L$ is also invariant. It follows that the restriction of $\varphi$ to $W$ fixes every line in $L$ and hence is a homothety of ratio $\lambda\in\R$: $\varphi(v)=\lambda v$ for all $v\in W$.

We claim that $\lambda=\pm 1$. In order to see that this is the case let $v_1,\ldots, v_r \in W$ with $v_1\wedge \cdots \wedge v_r\neq 0$ and compute
\begin{eqnarray*}
v_1\wedge \cdots \wedge v_r &=&(\wedge^r\varphi)(v_1\wedge\cdots \wedge v_r)=\varphi(v_1)\wedge\cdots \wedge \varphi(v_r)\\
&=&(\lambda v_1)\wedge\cdots \wedge(\lambda v_r) =\lambda^r(v_1\wedge \cdots \wedge v_r).
\end{eqnarray*}
The claim follows.
\end{proof}

In order to prove the second claim of Proposition \ref{prop:inv_H*} we are going to apply Lemma \ref{lemma:rigidity} to $V=H^1(X;\R)$, $W=H^1(X;\R)^H$ and $\varphi=g^*$ for $g\in G$. Notice that $\dim H^1(X;\R)^H=\dim H^1(\bT^n;\R)=n\ge r+1$. 

Consider the commutative diagram
\begin{equation}
\label{eq:diag}
\xymatrix{
\bigwedge^r(H^1(X;\R)^H)\ar[d]_{\wedge^r\pi^*}\ar[r]^{\wedge} & H^r(X;\R)^H\ar[d]_{\pi^*} \\
\bigwedge^r(H^1(X/H;\R))\ar[r]^{\wedge} & H^r(X/H;\R)
}
\end{equation}
where the vertical arrows are the pull-back of the orbit map $\pi\colon X\to X/H$, while the horizontal maps are the cup product in cohomology. By Proposition \ref{prop:invco}, the two vertical arrows are isomorphisms. On the other hand, the lower horizontal arrow is also an isomorphism because $X/H=\bT^n$ is a $n$-dimensional torus. This proves that the homomorphism 
$$\bigwedge{ }^r\left(H_1(X;\R)^H\right)\to H^r(X;\R)^H$$ 
is an isomorphism. 

By the first claim of Proposition \ref{prop:inv_H*} we have that $g^*$ acts trivially on $H^r(X;\R)^H$ and hence on $\bigwedge^r\left(H_1(X;\R)^H\right)$. In other words, Lemma \ref{lemma:rigidity} applies and the second claim of Proposition \ref{prop:inv_H*} follows. This concludes the proof of Proposition \ref{prop:inv_H*} .\qed
\medskip

Let $G \to \mathrm{Hom}(H^1(X;\R)^H,H^1(X;\R)^H)$ be the map $g\mapsto g^*$. By Proposition \ref{prop:inv_H*}, the image of this homomorphism is $\{\pm \id\}$. Identifying the latter group with  the multiplicative group $\bF=\{\pm 1\}$ of two elements, we obtain a homomorphism 
\begin{equation}\label{eq-sign}
\delta\colon G\to \bF.
\end{equation}
Notice that by definition $H\subset\Ker(\delta)$.

With the same notation as above we can now show that $H^*(X;\R)^H$ is in fact $G$-invariant and describe the action 
$$G\curvearrowright H^*(X;\R)^H.$$

\begin{corollary}\label{cor-inv}
For $s=0,1,\dots,n$ and $[\omega]\in H^s(X;\R)^H$ we have
$$g^*[\omega]=\delta(g)^s[\omega]$$ 
for all $g\in G$. 
\end{corollary}
\begin{proof}
Using the same argument as in the proof of Proposition \ref{prop:inv_H*}, more concretely using diagram \eqref{eq:diag} and the discussion following that diagram, we deduce that the standard homomorphism
$$\bigwedge{ }^s\left(H_1(X;\R)^H\right)\to H^s(X;\R)^H$$ 
is an isomorphism. By Proposition \ref{prop:inv_H*}, and by definition of $\delta(\cdot)$, we have that $g$ acts on $H^1(X;\R)^H$ by multiplication by $\delta(g)$. Hence, $g$ acts on $H^s(X;\R)^H$ by multiplication by $\delta(g)^s$, as we claimed.
\end{proof}

>From here we obtain the first non-trivial restrictions on any manifold $N$ as in (*).

\begin{corollary}\label{cor-cool}
With the same notation as in (*) the following holds:
\begin{enumerate}
\item The pullback $f^*\colon H^{2*}(N;\R)\to H^{2*}(\bT^n;\R)$ is an isomorphism. 
\item If $n$ or $r$ is odd, then $f^*\colon H^*(N;\R)\to H^*(\bT^n;\R)$ is an isomorphism.
\end{enumerate}
\end{corollary}
\begin{proof}
Since $H^s(N;\R)=H^s(X;\R)^G\subset H^s(X;\R)^H$ we deduce from Corollary \ref{cor-inv} that 
$$H^s(N;\R)=\left(H^s(X;\R)^H\right)^G$$
for $s$ even.
Also by Corollary \ref{cor-inv} we have that the action $G\curvearrowright H^s(X;\R)^H$ is trivial for even $s$; hence
$$H^s(N;\R)=\left(H^s(X;\R)^H\right)^G=H^s(X;\R)^H=H^s(\bT^n;\R)$$
for even $s$. This proves (1).

Suppose now, for the sake of concreteness, that $n$ is odd. Notice that on the one hand $g^*$ is trivial on $H^n(X;\R)$ by Proposition \ref{polyhedron} while on the other hand it acts by multiplication by $\delta(g)^n$ by Corollary \ref{cor-inv}. Hence $\delta(g)=1$ for all $g\in G$. Observe that the same argument applies is $r$ is odd. Once we know that $\delta(g)=1$ for all $g\in G$, the second claim follows as (1).
\end{proof}

\section{Averaging}\label{sec:averaging}
In this section we construct $n$ cocycles in $C^1(X;\R)$ which represent a basis of $H^1(X;\R)^H$ and have some highly desirable properties; notation is as in (*). 

To begin with choose bases $[c_1],\dots,[c_n]$ of $H_1(\bT^n;\Z)$ and $[\theta_1],\dots,[\theta_n]$ of $H^1(\bT^n;\R)$ with
$$\langle[\theta_i],[c_j]\rangle=\delta_{ij},$$
where $\delta_{ij}$ is the Kronecker symbol.

\begin{proposition}\label{prop:cocycles}
There are 1-cocycles $\Theta_1,\dots,\Theta_n\in C^1(X;\R)$ with the following properties:
\begin{enumerate}
\item $\pi^*([\theta_i])=[\Theta_i]$ for $i=1,\dots,n$.
\item $g^\#(\Theta_i)=\delta(g)\Theta_i$ for any $g\in G$ and $i=1,\dots,n$. Here $\delta(\cdot)$ is the homomorphism \eqref{eq-sign}.
\item $\langle\Theta_i,\tilde c_j\rangle=\delta_{ij}$ for any 1-chain $\tilde c_j\in C_1(X;\R)$ in $X$ such that $\pi_\#(\tilde c_j)$ is 1-cycle representing $[c_j]$.
\end{enumerate}
\end{proposition}

We begin choosing 1-cocycles $\theta_1,\dots,\theta_n\in C^1(\bT^n;\R)$ representing the cohomology classes $[\theta_1],\dots,[\theta_n]$ and define $\Theta_i$ as a twisted $G$-average of the pull-back $\pi^\#(\theta_i)$:
$$\Theta_i=\frac 1{\vert G\vert}\sum_{g\in G}\delta(g)g^\#(\pi^\#(\theta_i)).$$
Observe that $\Theta_i$ is a cocycle because it is a weighted sum of the cocycles $g^\#(\pi^\#(\theta_i))$. We claim that the cocycles $\Theta_1,\dots,\Theta_n$ fulfill the claims of Proposition \ref{prop:cocycles}.

To show (1), notice that, since $\pi^\#(\theta_i)$ represents the cohomology class $\pi^*([\theta_i])\in \pi^*(H^1(X/H;\R))=H^1(X;\R)^H$, we have by Proposition \ref{prop:inv_H*} that 
$$[g^\#(\pi^\#(\theta_i))]=g^*[\pi^\#(\theta_i)]=\delta(g)\pi^*([\theta_i]).$$
Since $\delta(g)^2=1$ for all $g\in G$ we have
$$[\Theta_i]=\left[\frac 1{\vert G\vert}\sum_{g\in G}\delta(g)g^\#(\pi^\#(\theta_i))\right]=
\frac 1{\vert G\vert}\sum_{g\in G}\delta(g)^2\pi^*([\theta_i])=\pi^*[\theta_i]$$
as claimed.

The validity of (2) follows from a simple computation and the fact that $\delta$ is a homomorphism.

In order to show (3) observe that the homomorphism $\delta$ from \eqref{eq-sign} induces a well-defined map
$$\delta:G/H\to\bF$$
because $H\subset\Ker(\delta)$. From now on we choose a representative $g$ for every class $gH\in G/H$ and hence identify $G/H$ with a subset of $G$. We start rewriting $\langle\Theta_i,\tilde c_j\rangle$:

\begin{equation}
\label{eq:iterated_average}
\begin{split}
\langle\Theta_i,\tilde c_j\rangle 
   & = \frac 1{\vert G\vert}\left\langle\sum_{g\in G}\delta(g)g^\#(\pi^\#(\theta_i)),\tilde c_j\right\rangle \\
   & = \frac 1{\vert G\vert}\left\langle\pi^\#(\theta_i),\sum_{g\in G}\delta(g)g_\#(\tilde c_j)\right\rangle \\
   & = \frac 1{\vert G\vert}\left\langle\pi^\#(\theta_i),\sum_{g\in G/H} \sum_{h \in H} \delta(gh) (gh)_\#(\tilde c_j)\right\rangle \\
   & = \frac 1{\vert G\vert}\left\langle\pi^\#(\theta_i),\sum_{g\in G/H} \sum_{h \in H} \delta(g) g_\#(h_\#(\tilde c_j))\right\rangle \\   
   & = \frac 1{\vert G\vert}\left\langle\pi^\#(\theta_i),\sum_{g\in G/H} \delta(g)g_\#\left(\sum_{h \in H} h_\#(\tilde c_j)\right)\right\rangle \\
   & =  \frac 1{\vert G\vert}\left\langle\sum_{g\in G/H} \delta(g)g^\#(\pi^\#(\theta_i)),\sum_{h \in H} h_\#(\tilde c_j)\right\rangle.
\end{split}
\end{equation}
Notice that, by Lemma \ref{lem:transfer}, $\sum_{h \in H} h_\#(\tilde c_j)$ is a 1-cycle satisfying
\begin{equation}\label{eq-bla3}
\pi_*\left(\left[\sum_{h \in H} h_\#(\tilde c_j)\right]\right)=\vert H\vert[c_j].
\end{equation}
Observing that in the last equation in \eqref{eq:iterated_average} we are evaluating a 1-cocycle and a 1-cycle we obtain the same result if we evaluate the corresponding cohomology and homology classes:

\begin{equation}
\label{eq:iterated_average2}
\begin{split}
\langle\Theta_i,\tilde c_j\rangle 
   & =  \frac 1{\vert G\vert}\left\langle\sum_{g\in G/H} \delta(g)g^\#(\pi^\#(\theta_i)),\sum_{h \in H} h_\#(\tilde c_j)\right\rangle\\
   & =  \frac 1{\vert G\vert}\left\langle\left[\sum_{g\in G/H} \delta(g)g^\#(\pi^\#(\theta_i))\right],\left[\sum_{h \in H} h_\#(\tilde c_j)\right]\right\rangle\\
   & =  \frac 1{\vert G\vert}\left\langle\sum_{g\in G/H} \delta(g)g^*(\pi^*([\theta_i])),\left[\sum_{h \in H} h_\#(\tilde c_j)\right]\right\rangle.
\end{split}
\end{equation}
By Proposition \ref{prop:inv_H*}, and $\delta(g)^2=1$, we have 
\begin{equation}\label{eq-bla2}
\delta(g)g^*(\pi^*([\theta_i]))=\pi^*[\theta_i]
\end{equation}
for all $g\in G$. Therefore, we obtain from \eqref{eq:iterated_average2} and \eqref{eq-bla2}:
\begin{equation}
\label{eq:iterated_average3}
\begin{split}
\langle\Theta_i,\tilde c_j\rangle 
   & =  \frac 1{\vert G\vert}\left\langle\sum_{g\in G/H} \delta(g)g^*(\pi^*([\theta_i])),\left[\sum_{h \in H} h_\#(\tilde c_j)\right]\right\rangle \\
   & = \frac {\vert G/H\vert}{\vert G\vert}\left\langle\pi^*([\theta_i]),\left[\sum_{h \in H} h_\#(\tilde c_j)\right]\right\rangle \\
      & = \frac {\vert G/H\vert}{\vert G\vert}\left\langle[\theta_i],\pi_*\left[\sum_{h \in H} h_\#(\tilde c_j)\right]\right\rangle \\
    & = \frac{\vert G/H\vert}{\vert G\vert}\left\langle[\theta_i],\vert H\vert[c_j]\right\rangle,
\end{split}
\end{equation}
where the last equality holds by \eqref{eq-bla3}. It follows that
$$\langle\Theta_i,\tilde c_j\rangle=\langle[\theta_i],[c_j]\rangle=\delta_{ij}$$
as claimed. This concludes the proof of Proposition \ref{prop:cocycles}.\qed
\medskip

Before moving on we observe that evaluating the cocycles $\Theta_i$ provided by Proposition \ref{prop:cocycles} onto any integral cycle in $X$ we obtain not only a real number but in fact an integer.

\begin{lemma}\label{lemma-aj}
Let $c\in C_1(X;\Z)$ be an integral 1-chain whose push-forward $\pi_\#(c)$ is a 1-cycle in $X/H=\bT^n$. Then $\langle\Theta_i,c\rangle\in\Z$.
\end{lemma}
\begin{proof}
A similar computation as in the proof of (3) in Proposition \ref{prop:cocycles} shows that
$$\langle\Theta_i,c\rangle=\langle[\theta_i],[\pi_\#(c)]\rangle.$$
The claim follows because $[\pi_\#(c)]\in H_1(X/H,\Z)$ can be written as a linear combination of the basis $[c_1],\dots,[c_n]$ with integer coefficients and because $\langle[\theta_i],[c_j]\rangle=\delta_{ij}$ by the choice of the bases.
\end{proof}

\section{The Abel-Jacobi map}\label{sec:abel-jacobi}
Still with the same notation as in (*), we construct in this section a map $\Psi:X\to\R^n/\Z^n$ and study its equivariance properties with respect to the action of $G$ on $X$. The construction of the map $\Psi$ is motivated by the construction of the classical Abel-Jacobi map of a Riemann surface into its Jacobian.
\medskip

To begin with we consider the linear map
$$\tilde\Psi\colon C_1(X;\R)\to\R^n,\ \tilde\Psi(c)=(\langle\Theta_i,c\rangle)_{i=1,\dots,n},$$
given by evaluating the the 1-cocycles $\Theta_1,\dots,\Theta_n$ in Proposition \ref{prop:cocycles} on singular chains. By Lemma \ref{lemma-aj}, the image of $Z_1(X;\Z)$ under $\tilde\Psi$ is contained in $\Z^n$. In particular, $\tilde\Psi$ induces a map
$$\Psi':C_1(X;\R)/Z_1(X;\Z)\to\R^n/\Z^n$$
In order to obtain a map defined on $X$ we choose a base point $x_0$. Suppose that $x\in X$ is another point and that $I$ and $I'$ are two paths starting at $x_0$ and ending in $x$. Considering both paths as 1-chains we have $\Psi'(I)=\Psi'(I')$ because $I-I'$ is a 1-cycle. It follows that the value $\Psi'(I)$ does not depend on the chosen path $I$ but only on the point $x$. We obtain hence a well-defined map
$$\Psi:X\to\R^n/\Z^n.$$
We will refer to $\Psi$ as {\em the Abel-Jacobi map}. 

We start investigating how the action of $G$ on $X$ is reflected by the Abel-Jacobi map. In order to do so we consider $\R^n/\Z^n$ as an abelian group.

\begin{lemma}\label{lem:affine}
For all $g\in G$ there is $A_g\in\R^n/\Z^n$ with 
$$\Psi(g(x))=\delta(g)\Psi(x)+A_g$$
for all $x\in X$.
\end{lemma}
\begin{proof}
For $x\in X$ denote by $I_x$ a path from $x_0$ to $x$. In order to compute $\Psi(g(x))$ we consider the juxtaposition $I_{g(x_0)}+g(I_x)$ of $I_{g(x_0)}$ and $g(I_x)$. We have thus
$$\Psi(g(x))=\Psi'(I_{g(x_0)}+g(I_x))=\Psi'(I_{g(x_0)})+\Psi'(g(I_x))=\Psi(g(x_0))+\Psi'(g(I_x)).$$
We set $A_g=\Psi(g(x_0))$ and compute
$$\Psi'(g(I_x))=(\langle\Theta_i,gI_x\rangle)_{i=1,\dots,n}.$$
We have for $i=1,\dots,n$
$$\langle\Theta_i,g(I_x)\rangle=\langle\Theta_i,g_\#(I_x)\rangle=\langle g^\#\Theta_i,I_x\rangle=
\delta(g)\langle\Theta_i,I_x\rangle,$$
where the last equation holds by Proposition \ref{prop:cocycles}. The claim follows.
\end{proof}

Clearly, the map
\begin{equation}\label{eq-action}
G\times\R^n/\Z^n\to\R^n/\Z^n,\ \ (g,x)\mapsto\delta(g) x+A_g,
\end{equation}
is an action of $G$ on $\R^n/\Z^n$. We will see next that this action is trivial when restricted to the subgroup $H$. 

Notice that for any $h\in H$ the path $I_{h(x_0)}$ projects to a closed loop in $X/H$. In other words, $\pi_\#(I_{h(x_0)})$ is a 1-cycle in $X/H=\bT^n$. It follows from Lemma \ref{lemma-aj} that $\tilde\Psi(I_{h(x_0)})\in\Z^n$, meaning that
$$A_h=\tilde\Psi(I_{h(x_0)})\mod\Z^n=0.$$
Since $H\subset\Ker(\delta)$ we deduce from Lemma \ref{lem:affine} that in fact the action of $H$ on $\R^n/\Z^n$ is trivial.

\begin{lemma}\label{lem:trivial-H}
The restriction of the $G$-action \eqref{eq-action} to $H$ is trivial.\qed
\end{lemma} 

Lemma \ref{lem:trivial-H} implies directly that the Abel-Jacobi map $\Psi$ descends to a map
$$\hat\Psi:X/H\to\R^n/\Z^n.$$
We claim now that $\hat\Psi$ is a homotopy equivalence. In order to see that this is the case let $c_1$ be a loop in $X/H=\bT^n$ representing the homology class $[c_1]\in H_1(X/H;\Z)=\pi_1(X/H)$ chosen in the previous section. Suppose also that $c_1$ is based at $\pi(x_0)$ and avoids $Z_f$, the preimage of the branching locus of $f$. This last condition implies that there is a unique lift $\tilde c_1$ of $c_1$ starting at $x_0$. Applying Proposition \ref{prop:cocycles} (3) to $\tilde c_1$ we obtain that 
$$\tilde\Psi(\tilde c_1)=(1,0,\ldots,0).$$
Identifying $\pi_1(\R^n/\Z^n)$ with the deck-transformation group $\Z^n$ of the cover $\R^n\to\R^n/\Z^n$ we have that the image $\hat\Psi(c_1)$ of $c_1$ is homotopic to the first standard generator of $\pi_1(\R^n/\Z^n)\simeq\Z^n$. Arguing in the same way for $c_2,\ldots, c_n$ we have proved that the homomorphism
$$\pi_1(\hat\Psi):\pi_1(X/H,\pi(x_0))\to\pi_1(\R^n/\Z^n,\hat\Phi(\pi(x_0)))$$
induced by $\hat\Psi$ maps the basis $[c_1],\dots,[c_n]$ to the standard basis of $\Z^n=\pi_1(\R^n/\Z^n)$. In other words, $\pi_1(\hat\Psi)$ is an isomorphism. Since both $X/H$ and $\R^n/\Z^n$ are tori, and hence have contractible universal coverings, this implies that $\Psi\colon X/H\to\R^n/\Z^n$ is a homotopy equivalence, as we wanted to show.

\begin{proposition}\label{prop:descend}
The Abel-Jacobi map $\Psi:X\to\R^n/\Z^n$ descends to a homotopy equivalence $\hat\Psi\colon\bT^n=X/H\to\R^n/\Z^n$.\qed
\end{proposition}

We can now start drawing a diagram to which we will continue adding arrows and objects in the next section:

$$\xymatrix{
 & X\ar[dl]_{\Psi}\ar[d]_\pi\ar[dr]^{\pi_G} \\
 \R^n/\Z^n & X/H=\bT^n\ar[l]_{\hat\Psi}\ar[r]^f & X/G=N
}
$$

\section{The subgroup $K\subset G$}\label{sec:pushdown}
Still with the same notation as in (*) define 
$$K=\Ker(\delta)$$ 
to be the kernel of the homomorphism $\delta$; see \eqref{eq-sign}. By construction $K$ is a normal subgroup of $G$ with at most index $2$. Notice that $H\subset K=\Ker(\delta)$ and that it follows from Corollary \ref{cor-inv} that $K=G$ if the dimension $n$ of the involved manifolds is odd.

\begin{remark}
In the next section, we will also prove that $K=G$ if $n$ is even.
\end{remark}

By Lemma \ref{lem:affine}, the action \eqref{eq-action} of $G$ on $\R^n/\Z^n$ restricts to an action of $K$ by translations. Observe that since $H\subset K$, this action is not effective by Lemma \ref{lem:trivial-H}. In spite of this, we will denote by $(\R^n/\Z^n)/K$ the quotient of $\R^n/\Z^n$ under the action by  translations of $K$. In this section we prove:

\begin{proposition}\label{prop:descend2}
The homotopy equivalence $\hat\Psi:X/H\to\R^n/\Z^n$ descends to a map $\bar\Psi:X/K\to(\R^n/\Z^n)/K$. In particular, the following holds:
\begin{enumerate}
\item The orbit map $p_1:\R^n/\Z^n\to(\R^n/\Z^n)/K$ has degree $\vert K/H\vert$.
\item $H$ is a normal subgroup of $K$.
\item The action of $K/H$ on $\R^n/\Z^n$ is effective and $(\R^n/\Z^n)/K$ is a torus.
\item The action of $K/H$ on $X/H$ is free. In particular, the space $X/K$ is in fact a manifold and the orbit map $p_2:X/H\to X/K$ is a covering map.
\item The map $\bar\Psi:X/K\to(\R^n/\Z^n)/K$ is a homotopy equivalence. 
\end{enumerate}
\end{proposition}
\begin{proof}
The fact that $\hat\Psi:X/H\to\R^n/\Z^n$ descends to a map $\bar\Psi:X/K\to(\R^n/\Z^n)/K$ follows just from the definitions and is left to the reader. Notice that since $K$ acts on the $n$-dimensional torus $\R^n/\Z^n$ by translations, the quotient is $n$-dimensional a torus as well. This proves the second part of (3).

At this point we can enlarge the diagram above as follows:
\begin{equation}\label{eq-diag}
\xymatrix{
 & X\ar[dl]_{\Psi}\ar[d]_\pi\ar[ddr]^{\pi_G} \\
 \R^n/\Z^n\ar[d]^{p_1} & X/H=\bT^n\ar[l]_{\hat\Psi}\ar[dr]^f\ar[d]^{p_2} & \\
 (\R^n/\Z^n)/K & X/K\ar[l]_{\bar\Psi}\ar[r]^{p_3} & X/G=N
}
\end{equation}
where $p_1$, $p_2$ and $p_3$ are the obvious orbit maps.

Since the square in the left bottom corner of the diagram commutes we have that
$$\deg(p_1)\deg(\hat\Psi)=\deg(p_1\circ\hat\Psi)=\deg(\bar\Psi\circ p_2)=\deg(\bar\Psi)\deg(p_2).$$
Taking into account that $\hat\Psi$ is a homotopy equivalence and that $p_2$ has degree $\vert K/H\vert$ we obtain
$$\deg(p_1)=\vert K/H\vert\deg(p_2).$$
On the other hand, $p_1$ has positive degree. This implies that $p_2$ has to have positive degree as well; hence
$$\deg(p_1)\ge \vert K/H\vert.$$
On the other hand, $\deg(p_1)\le \vert K/H\vert$ because the subgroup $H$ of $K$ acts trivially on $\R^n/\Z^n$. We have proved (1).

Recall that $H$ is contained in the kernel of the homomorphism 
$$K\to\Aut(\R^n/\Z^n)$$
by Lemma \ref{lem:trivial-H}. The image of this homomorphism has $\deg(p_1)=\vert K/H\vert$ elements. It follows that $H$ is in fact precisely the kernel of this homomorphism and that it is therefore normal in $K$. We have proved (2) and the first part of (3).

Since $H$ is normal in $K$ we have that $K/H$ acts on $X/H$ and that $X/K=(X/H)/(K/H)$. We claim that the action of $K/H$ on $X/H$ is free. In order to see this, suppose that we have $x\in X/H$ and $k\in K/H$ with $kx=x$. Then we have 
$$\hat\Psi(x)=\hat\Psi(kx)=k\hat\Psi(x).$$
Since the action of $K/H$ on $\R^n/\Z^n$ is effective and by translations, an element $k\in K/H$ can only fix the point $\Psi(x)\in\R^n/\Z^n$ if $k$ is the neutral element in $K/H$. We have proved (4).

It remains to show that the map $\bar\Psi:X/K\to(\R^n/\Z^n)/K$ is a homotopy equivalence. As in the proof of Proposition \ref{prop:descend}, it suffices to show that $(\pi_1)_*(\bar\Psi):\pi_1(X/K)\to\pi_1((\R^n/\Z^n)/K)$ is an isomorphism. Notice at this point that we have the following commutative diagram:
$$\xymatrix{
1 \ar[r] & \pi_1(X/H) \ar[d]_{(\pi_1)_*(\hat\Psi)}\ar[r] & \pi_1(X/K) \ar[d]_{(\pi_1)_*(\bar\Psi)}\ar[r] & K/H\ar[d]_{\id} \ar[r] & 1 \\
1\ar[r] & \pi_1(\R^n/\Z^n) \ar[r] & \pi_1((\R^n/\Z^n)/K) \ar[r] & K/H \ar[r] & 1
}$$
Since the first and third vertical arrows are isomorphisms it follows that the middle arrow is an isomorphism as well. This concludes the proof of Proposition \ref{prop:descend2}.
\end{proof}

We conclude this section with some remarks needed to prove Theorem \ref{thm:main} below. The following is just a direct consequence of parts (3) and (5) of Proposition \ref{prop:descend2}:

\begin{corollary}\label{cor:X/G}
The manifold $X/K$ is homotopy equivalent to a $n$-dimensional torus. In particular, $\pi_1(X/K)=H_1(X/K;\Z)=\Z^n$ and $H^*(X/K;\R)=H^*(\bT^n;\R)$.\qed
\end{corollary}

Recall now that $K\subset G$ is normal of at most index $2$ because it is the kernel of the homomorphism \eqref{eq-sign}. In particular, the group $G/K$ acts on $X/K$ with $X/G=(X/K)/(G/K)$ and the map $p_3:X/K\to X/G=N$ in \eqref{eq-diag} is just the orbit map. We prove now:

\begin{lemma}\label{lem:either-or}
If $K\neq G$ denote by $\sigma$ the non-trivial element in $G/K$. Then $\sigma$ acts as $-\id$ on $H^1(X/K;\R)$ and hence $H^1(N;\R)=0$.
\end{lemma}

We suggest the reader to compare with Corollary \ref{cor-cool}.

\begin{proof}
Denote by $\tilde\sigma\in G$ a representative of $\sigma$ in $G$. Notice that, by the very definition of the action \eqref{eq-action}, $\tilde\sigma$ acts on $H^1(\R^n/\Z^n;\R)$ by $-\id$. Since $K$ acts on $\R^n/\Z^n$ by translations, it follows that $\sigma$ also acts on $H^1((\R^n/\Z^n)/K;\R)$ by $-\id$. Since the homotopy equivalence $\bar\Psi$ satisfies $\bar\Psi(\sigma(x))=\sigma(\bar\Psi(x))$, we deduce that $\sigma$ acts on $H^1(X/K;\R)$ as $-\id$ as well. This proves the first claim. In order to prove the second one observe that
$$H^1(N;\R)=H^1((X/K)/(G/K);\R)=H^1(X/K;\R)^{G/K}=0.$$
This concludes the proof of Lemma \ref{lem:either-or}.
\end{proof}

\section{The final step of the proof of Theorem \ref{thm:main}}\label{sec:final}
In this section we conclude the proof of Theorem \ref{thm:main}; as always, notation is as in (*). We copy \eqref{eq-diag} here for the convenience of the reader:
 $$\xymatrix{
 & X\ar[dl]_{\Psi}\ar[d]_\pi\ar[ddr]^{\pi_G} \\
 \R^n/\Z^n\ar[d]^{p_1} & X/H=\bT^n\ar[l]_{\hat\Psi}\ar[dr]^f\ar[d]^{p_2} & \\
 (\R^n/\Z^n)/K & X/K\ar[l]_{\bar\Psi}\ar[r]^{p_3} & X/G=N
}
$$
So far, we know that $f=p_3\circ p_2$ and that $p_2$ is a covering map. In order to prove that $f$ is itself a covering map it suffices to show that $p_3=\id$. Since $p_3$ is the orbit map $X/K\to(X/K)/(G/K)$ it suffices to show:

\begin{proposition}\label{final-prop}
$G/K$ is the trivial group.
\end{proposition}

Assuming Proposition \ref{final-prop} we conclude the proof of Theorem \ref{thm:main}:

\begin{named}{Theorem \ref{thm:main}}
Let $N$ be a closed, connected, and oriented $n$-manifold, $n\ge 2$ so that $\dim H^r(N;\R)=\dim H^r(\bT^n;\R)$ for some $1\le r < n$.
Then every branched cover $f\colon \bT^n \to N$ is a cover. In particular, every $\pi_1$-surjective branched cover $\bT^n \to N$ is a homeomorphism.
\end{named}

\begin{proof}
By Proposition \ref{final-prop} we have $K=G$ and hence $p_3=\id$. In particular, we have 
$$f=p_3\circ p_2=p_2.$$
By Proposition \ref{prop:descend2}, the map $p_2$, and hence $f$, is a covering map.
\end{proof}

It remains to prove the proposition:

\begin{proof}[Proof of Proposition \ref{final-prop}] 
Seeking a contradiction assume that $K\neq G$. Hence $G/K$ is the group of order two. As in Lemma \ref{lem:either-or}, denote by $\sigma$ the non-trivial element in $G/K$. 

We are going to derive a contradiction to $K\neq G$ studying the fixed point set $\Fix(\sigma)$ of $\sigma$. It follows from Corollary \ref{cor:X/G} and Lemma \ref{lem:either-or} that $\sigma$ acts on $H^s(X/K;\R)\simeq H^s(\bT^n;\R)$ as multiplication by $(-1)^s$. In particular, $\sigma:X/K\to X/K$ has Lefschetz number $L(\sigma)=2^n$. It follows from the Lefschetz fixed point theorem that 
\begin{equation}\label{eq-fixne}
\Fix(\sigma)\neq\emptyset.
\end{equation}

We prove now that the inclusion of $\Fix(\sigma)$ into $X/K$ is trivial on homology:

\begin{claim}\label{homology-trivial}
The map $H_1(\Fix(\sigma);\R)\to H_1(X/K;\R)$ is trivial.
\end{claim}
\begin{proof}
Seeking a contradiction suppose that $[\alpha]\in H_1(\Fix(\sigma);\R)$ is not trivial in $H_1(X/K;\R)$ and let $[\beta]\in H^{n-1}(X/K;\R)$ be the unique cohomology class with
$$\langle[\eta],[\alpha]\rangle=[\eta]\wedge[\beta]$$
for all $[\eta]\in H^1(X/K;\R)$. Since $[\alpha]$ is fixed by $\sigma$ we deduce that $[\beta]$ is also fixed by $\sigma$. This implies that
$$H^{n-1}(N;\R)=H^{n-1}((X/K)/\langle\sigma\rangle;\R)=H^{n-1}(X/K;\R)^{\langle\sigma\rangle}\neq 0$$
contradicting that $H^{n-1}(N;\R)\simeq H^1(N;\R)=0$ by Lemma \ref{lem:either-or}.
\end{proof}

By Corollary \ref{cor:X/G}, we can identify $H_1(X/K;\Z)$, $\pi_1(X/K)$, and $\Z^n$. In particular, the homomorphism $\pi_1(X/K)\to H_1(X/K;\R)$ is injective. Given now a connected component $\Sigma$ of $\Fix(\sigma)$ we have the following commutative diagram:
$$\xymatrix{
\pi_1(\Sigma)\ar[d]\ar[r] & \pi_1(X/K)\ar[d] \\
H_1(\Sigma;\R)\ar[r] & H_1(X/K;\R)
}$$
As we just observed the right vertical arrow is injective. On the other hand, the lower horizontal arrow is trivial by Claim \ref{homology-trivial}. This implies that the homomorphism $\pi_1(\Sigma)\to\pi_1(X/K)$ is trivial as well. In other words, $\Sigma$ lifts homeomorphically to the universal cover $\widetilde{X/K}$ of $X/K$. We have proved:

\begin{claim}\label{homotopy-trivial}
Every connected component $\Sigma$ of $\Fix(\sigma)$ lifts homeomorphically to the universal cover $\widetilde{X/K}$ of $X/K$.\qed
\end{claim}

Let now $\Sigma\subset\Fix(\sigma)$ be a connected component and lift it to $\tilde\Sigma\subset\widetilde{X/K}$. Observe that since $X/K$ is covered by $\bT^n=X/H$, we can identify $\widetilde{X/K}$ with the euclidean space. Choosing a point $\tilde x\in\tilde\Sigma$ there is a unique lift $\tilde\sigma$ of $\sigma$ to $\widetilde{X/K}$ with $\tilde\sigma(\tilde x)=\tilde x$. Observe that $\tilde\sigma$ is an involution and that 
$$\tilde\Sigma\subset\Fix(\tilde\sigma).$$
The involution $\tilde\sigma$ is obviously a transformation of the contractible space $\widetilde{X/K}$ of prime period $2$. It is a standard consequence of Smith theory \cite[III. Theorem 5.2]{Bredon} that $\Fix(\tilde\sigma)$ is $\Z_2$-acyclic and hence connected. This proves that $\Fix(\tilde\sigma)=\tilde\Sigma$. 

On the other hand, for $n\ge 3$, $\tilde\Sigma$ is homeomorphic to $\Sigma$ and hence homeomorphic to an orientable cohomology manifold that is not $\Z_2$-acyclic by Proposition \ref{fix-set}; this is a contradiction. For $n=2$, $\Fix(\sigma)$ is a discrete set. Thus \cite[Theorem III.4.3]{Bredon} gives a contradiction.

This concludes the proof of Proposition \ref{final-prop}.
\end{proof}

\section{Theorem \ref{no-sudoku}}\label{sec:no-sudoku}
We discuss now briefly Theorem \ref{no-sudoku}.

\begin{named}{Theorem \ref{no-sudoku}}
There is no branched cover from the 4-dimensional torus $\bT^4$ to $\#^3(\bS^2\times\bS^2)$, the connected sum of three copies of $\bS^2\times\bS^2$. On the other hand, there are, a fortiori $\pi_1$-surjective, maps $\bT^4\to \#^3(\bS^2\times\bS^2)$ of arbitrarily large degree.
\end{named}

The claim that there is no branched covering $\bT^4\to\#^3(\bS^2\times\bS^2)$ follows directly from Theorem \ref{thm:main}. Maps of arbitrarily large degree 
$$f:\bT^4\to\#^3(\bS^2\times\bS^2)$$
can either be constructed directly, or be shown to exist using the work of Duan and Wang \cite{Duan-Wang}. These authors prove namely the following remarkable theorem:

\begin{theorem*}[Duan-Wang]
Suppose $M$ and $L$ are closed oriented $4$-manifolds with intersection matrices $A$ and $B$ under given bases $\alpha$ for $H^2(M;\Z)$ and $\beta$ for $H^2(L;\Z)$, respectively. If $L$ is simply connected, then there is a map $f:M\to L$ of degree $k$ such that $f^*( \beta) = \alpha P$ if and only if $P^tAP=kB$.
\end{theorem*}

In \cite[Example 4]{Duan-Wang}, this theorem is applied in the particular case of maps $\bT^4\to \#^3(\bS^2\times \bS^2)$ to obtain mappings of every degree. We finish this section by recalling this example. Intersection forms of $H^2(\bT^2)$ and $H^2(\#^3(\bS^2\times \bS^2))$ in natural bases read as
\[
\oplus^3\left( \begin{array}{cc}
0 & 1 \\
1 & 0 
\end{array}\right).
\]
Thus the linear mapping, given by the matrix
\[
P= \oplus^3 \left( \begin{array}{cc}
0 & k \\
1 & 0 
\end{array}\right),
\]
yields the required result.


\bigskip

\noindent{\small Matematiikan ja tilastotieteen laitos, Helsingin yliopisto, Helsinki \newline \noindent
\texttt{pekka.pankka@helsinki.fi}}
\medskip

\noindent{\small Department of Mathematics, University of Michigan, Ann Arbor \newline \noindent
\texttt{jsouto@umich.edu}}

\end{document}